\documentclass[reqno]{amsart}

\usepackage[latin1]{inputenc}

\usepackage{enumerate}
\usepackage{amstext,amsmath,amsthm,amsfonts,amssymb,amscd}
\usepackage{latexsym,mathrsfs,eufrak,dsfont,euscript}

\usepackage{color}
\definecolor{rojo}{RGB}{221,0,0}

\usepackage{hyperref} % hace links referencias a todos los labels
\hypersetup{
    colorlinks,%
    citecolor=blue,%
    filecolor=red,%
    linkcolor=red,%
    urlcolor=blue
}

\usepackage{tikz}
\usetikzlibrary{arrows,calc}
\tikzset{
>=stealth',
help lines/.style={dashed, thick},
axis/.style={<->},
important line/.style={thick},
connection/.style={thick, dotted},
}

\usepackage{pstricks}
\usepackage{pst-plot}
\usepackage{pst-all}
\usepackage{graphicx}

\newtheorem{theorem}{Theorem}%[section]

\newtheorem{lemma}[theorem]{Lemma}

\theoremstyle{definition}

\theoremstyle{remark}

\numberwithin{equation}{section}

\DeclareMathOperator{\supp}{supp}
\newcommand{\abs}[1]{\left\vert#1\right\vert}

\newcommand{\norm}[1]{\left\Vert#1\right\Vert}

\def\zR{\ensuremath{\mathbb{R}}}

\def\pii{\left(}
\def\pdd{\right)}
\def\cii{\left[}
\def\cdd{\right]}

\def\tends{\rightarrow}

\allowdisplaybreaks
\begin{document}
\title []{Continuous time random walks and the Cauchy problem for the heat equation}
%
%

%%    Information for authors

\author[]{Hugo Aimar}
%\email{haimar@santafe-conicet.gov.ar}
%%
\author[]{Gast\'{o}n Beltritti}
%\email{gbeltritti@santafe-conicet.gov.ar}
%%
\author[]{Ivana G\'{o}mez}
%\email{ivanagomez@santafe-conicet.gov.ar}
%
\thanks{The research was supported  by CONICET, ANPCyT (MINCyT) and UNL}
%
%\subjclass[]{Primary  Secondary }
%
%\keywords{}
%

\begin{abstract}
In this paper we deal with anomalous diffusions induced by Continuous Time Random Walks - CTRW in $\mathbb{R}^n$. A particle moves in
$\mathbb{R}^n$ in such a way that the probability density function $u(\cdot,t)$ of finding it in region $\Omega$ of $\mathbb{R}^n$ is
given by $\int_{\Omega}u(x,t) dx$. The dynamics of the diffusion is provided by a space time probability density $J(x,t)$ compactly supported
in $\{t\geq 0\}$. For $t$ large enough, $u$ must satisfy the equation $u(x,t)=[(J-\delta)\ast u](x,t)$ where $\delta$ is the Dirac delta in space time.
We  give a sense to a Cauchy type problem for a given initial density distribution $f$. We use Banach fixed point method to solve it,
and we prove that under parabolic rescaling of $J$ the equation tends weakly to the heat equation and that for particular kernels $J$
the solutions tend to the corresponding temperatures when the scaling parameter approaches to zero.
\end{abstract}

\maketitle

\section{Introduction and statement of the results}

We shall be concerned with a probabilistic description of the motion of a particle in the space $\mathbb{R}^n$.
As usual we shall write $\mathbb{R}^{n+1}_{+}$ to denote the set $\{(x,t): x\in \mathbb{R}^{n} \textrm{ and } t\geq 0\}$.
Sometimes we shall also considerer the whole space time $\mathbb{R}^{n+1}=\{(x,t): x\in \mathbb{R}^{n} \textrm{ and } t\in \mathbb{R}\}$. The
$x$ variable is thought as a space variable, while $t$ represents time.

Let $u(x,t)$ denote, for $t$ fixed, the probability density of the position of the particle at time $t$. Precisely, for a given Borel set $E$ in $\mathbb{R}^{n}$ the quantity
$\mathcal{P}(t,E)=\int_E u(x,t) dx$ measures the probability of finding the particle in $E$ at time $t$.

The general problem is to find $u(x,t)$ when the dynamics of the system is known and some initial state is given.

Regarding the dynamics of the system we shall deal with anomalous diffusions. More precisely with continuous time random walks (CTRW). For a comprehensive introduction to the
subject we refer to \cite{MeKla2000}. A CTRW in  $\mathbb{R}^n$ is provided by a space-time probability density function, the kernel, $J(x,t)$ defined in $\mathbb{R}^{n+1}$. In this model
the particle has a probability density function $u(x,t)$ of arrival at position $x\in \mathbb{R}^n$ at time $t>0$ which depends on the events of arrival at any $y\in \mathbb{R}^n$,
sometimes only on the events of arrival at any $y$ in some neighborhood of $x$, at any previous time $s<t$. Precisely, this dependence is given by the convolution in $\mathbb{R}^{n+1}$
of $J$ with $u$ itself. In other words, for $t\geq 0$ and $x\in\mathbb{R}^{n}$

\begin{equation}\label{eq:asconvolutionJ}
u(x,t)=(J\ast u)(x,t)=\iint_{\mathbb{R}^{n+1}} J(x-y,t-s) u(y,s) dy ds.
\end{equation}

The physical condition of the dependence of the current position of the particle only on the past ($s<t$) gives us the first natural condition on $J$,

\textit{(J1)}\, $\supp J\subset \mathbb{R}^{n+1}_{+}$.

\noindent On the other hand, since $J$ is a density in $\mathbb{R}^{n+1}$, we must have

\textit{(J2)}\,$J\geq 0$, and

\textit{(J3)}\,$J\in L^1(\mathbb{R}^{n+1})$  and $\iint_{\mathbb{R}^{n+1}} J(x,t) dx dt =1$.

\medskip

Following the notation in \cite{MeKla2000}, the density function defined in $\mathbb{R}^{n}$ by
\begin{equation*}
\lambda(x)=\int_{\mathbb{R}} J(x,t) dt
\end{equation*}
is called \textit{the jump length probability function}. Notice that from \textit{(J1)}, $\lambda(x)=\int_{\mathbb{R}^+} J(x,t) dt$. On the other hand, \textit{the waiting time
probability function} is given by
\begin{equation*}
\tau(t)=\int_{\mathbb{R}^{n}} J(x,t) dx.
\end{equation*}

Regarding the initial condition, let us first assume that the particle is localized at the origin of $\mathbb{R}^{n}$ for $t<0$. In other words $u(x,t)=\delta_0(x)$ for $t<0$.
Hence, since $u(x,t)$ for $t\geq 0$ needs to satisfy \eqref{eq:asconvolutionJ}, from \textit{(J1)} we must have that
\begin{align*}
u(x,0) &= \iint_{\mathbb{R}^{n+1}} J(x-y,-s) u(y,s) dy ds\\
&= \int_{\mathbb{R}^{-}}\left(\int_{\mathbb{R}^{n}}J(x-y,-s)u(y,s) dy\right)ds\\
&= \int_{\mathbb{R}^{-}}\left(\int_{\mathbb{R}^{n}}J(x-y,-s)\delta_0(y) dy\right)ds\\
&= \int_{\mathbb{R}^{-}} J(x,-s) ds\\
&= \lambda(x).
\end{align*}
In other words, the deterministic situation, \textit{the particle is at the origin for $t<0$} produces \textit{immediately} at time $t=0$ a random situation modeled precisely
by the jump length probability function $\lambda(x)$ associated to the density $J$.

More generally, if the position at time $t<0$ of the particle distributes as indicates the density $f(x)$, then $u(x,0)=(\lambda\ast f)(x)$. In this framework the basic
initial problem we are interested in, takes the following form. Given $J(x,t)$ and $f(x)$, find $u(x,t)$ for $(x,t)\in\mathbb{R}^{n+1}_{+}$ such that

\begin{equation*}
(P)\left\{\begin{array}{cc} u(x,t)=(J\ast \overline{u})(x,t),\ \
x\in\mathbb{R}^{n}, t\geq 0;\\
\overline{u}(x,t)= \left\{\begin{array}{ccc}
&f(x),\  & t<0;\\
&u(x,t),\  & t\geq 0.
\end{array}
\right.
\end{array}
\right.
\end{equation*}

Sometimes, to emphasize the data $J$ and $f$ in \textit{(P)}, we shall write $P(J,f)$ for the problem $P$ and $u(J,f)$ for its solution.

Let us observe that the expected initial condition is attained since, taking $t=0$ in the first equation in \textit{(P)} we get $u(x,0)=(J\ast \overline{u})(x,0)=\iint J(x-y,-s)f(y)dy ds
=(\lambda\ast f)(x)$.

We shall consider wide families of kernels $J$, but there is one, the parabolic mean value kernels, which plays a more significant role for our
subsequent analysis. We shall use $\mathscr{H}$ (for heat) to denote these special occurrences of $J$. Let us introduce the most known of these kernels $\mathscr{H}$ (see \cite{Watson2012} or \cite{Evans98}). Set
$\mathcal{W}(x,t)$ to denote the Weierstrass kernel for $t>0$ and $x\in\mathbb{R}^{n}$. Precisely $\mathcal{W}(x,t)=(4\pi t)^{-n/2}e^{-\abs{x}^2/(4t)}$. Set
$E=\{(x,t)\in\mathbb{R}^{n+1}_{+}: \mathcal{W}(x,t)\geq 1\}$ and $\mathcal{H}(x,t)=\tfrac{1}{4}\mathcal{X}_E(x,t)\tfrac{\abs{x}^2}{t^2}$.

As it is easy to check $\mathcal{H}$ satisfies properties \textit{(J1)}, \textit{(J2)} and \textit{(J3)} stated above. Moreover, $\mathcal{H}$ satisfies also
the following two properties

\textit{(J4)}\, has compact support in $\mathbb{R}^{n+1}$;

\textit{(J5)}\, it is radial as a function of $x\in \mathbb{R}^{n}$ for each $t$.

\medskip

The outstanding fact regarding $\mathcal{H}$ is given by its role in the mean value formula for temperatures. If $v(x,t)$ is a solution
of the heat equation $\tfrac{\partial v}{\partial t}=\triangle v$ in a domain $\Omega$ in $\mathbb{R}^{n+1}$, then, for $(x,t)\in\Omega$ and $r$ small enough
we have that
$v(x,t)=\iint\mathcal{H}_r(x-y,t-s)v(y,s)dy ds$, where $\mathcal{H}_r$ denotes the parabolic $r$-mollifier of $\mathcal{H}$. Precisely
\begin{equation*}
\mathcal{H}_r(x,t)=\frac{1}{r^{n+2}}\mathcal{H}\left(\frac{x}{r},\frac{t}{r^2}\right)=\frac{1}{r^n}\mathcal{X}_{E(r)}(x,t)\frac{\abs{x}^2}{t^2},
\end{equation*}
with $E(r)=\{(x,t)\in \mathbb{R}^{n+1}_{+}: \mathcal{W}(x,t)\geq r^{-n}\}$. The following figure depicts the support
$E(r)$ of $\mathcal{H}_r$.

\begin{figure}[h!]
\psset{xunit=.7cm, yunit=.7cm}
\begin{pspicture}(5.5,-2.5)(-.6,2.5)
%\psgrid[griddots=10,gridlabels=2pt, subgriddiv=5, gridcolor=black]
  \psaxes[linecolor=gray, linewidth=.7pt, gridlabels=6pt, ticks=none,labels=none]{<->}(0,0)(5.5,-2.5)(-.6,2.5)
\psplot[linecolor=blue!60,linewidth=1.2pt]%
    {0.00001}{0.3183}{4 x mul 1  0.5 4 3.1415 mul x mul sqrt mul div ln mul sqrt}
  \psplot[linecolor=blue!60,linewidth=1.2pt]%
    {0.00001}{0.3183}{4 x mul 1  0.5 4 3.1415 mul x mul sqrt mul div ln mul sqrt neg}
\psplot[linecolor=blue!60,linewidth=1.2pt]%
    {0.00001}{1.2732}{4 x mul 1  0.25 4 3.1415 mul x mul sqrt mul div ln mul sqrt}
    \psplot[linecolor=blue!60,linewidth=1.2pt]%
    {0.00001}{1.2732}{4 x mul 1  0.25 4 3.1415 mul x mul sqrt mul div ln mul sqrt neg}
\psplot[linecolor=blue!60,linewidth=1.2pt]%
    {0.00001}{5.0929}{4 x mul 1  0.125 4 3.1415 mul x mul sqrt mul div ln mul sqrt}
    \psplot[linecolor=blue!60,linewidth=1.2pt]%
    {0.00001}{5.0929}{4 x mul 1  0.125 4 3.1415 mul x mul sqrt mul div ln mul sqrt neg}
\rput[c](5.4,-0.3){$t$}
\rput[c](-0.3,2.5){$x$}
\end{pspicture}
\caption{Sets $E(r)$ for $n=1$ and $r=\tfrac{1}{2}, \tfrac{1}{4}, \tfrac{1}{8}$.}
\label{fig:heatballs}
\end{figure}
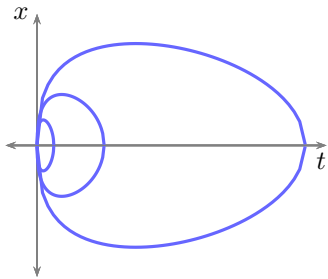

\medskip
In the sequel, for any kernel $J(x,t)$ and $r>0$ we shall write $J_r(x,t)$ to denote the parabolic approximation to the identity given by $J_r(x,t)=\frac{1}{r^{n+2}}J\left(\frac{x}{r},\frac{t}{r^2}\right)$. Moreover, the notation $v_r(x,t)$ or even $f_r(x)$ for functions depending on space time or, only on the space variable will have always the same meaning. Precisely, $v_r(x,t)=r^{-n-2}v(r^{-1}x,r^{-2}t)$ and $f_r(x)=r^{-n-2}f(r^{-1}x)$.

\medskip
The results of this paper are in the spirit of those in \cite{CERW08} and \cite{CoElRo09}. Instead of dealing with generalization of boundary conditions, we are concerned with diffusion problems in the whole space $\mathbb{R}^{n}$ and the initial condition is generalized.

\medskip
Let us state the main results of this paper. The first one is the weak convergence to the heat equation.

\begin{theorem}\label{thm:theorem1limits}
Assume that $J(x,t)$ satisfies (J2), (J3), (J4) and (J5). Then, for each $\varphi$ in the Schwartz class of $\mathbb{R}^{n+1}$,
we have
\begin{equation*}
\lim_{r\to 0}\frac{1}{r^2}\left(J_r-\delta\right)\ast\varphi = \mu\frac{\partial \varphi}{\partial t} + \nu\triangle\varphi,
\end{equation*}
uniformly on $\mathbb{R}^{n+1}$, where $\mu=-\iint t J(x,t) dx dt$ and $\nu=\frac{1}{2n}\iint \abs{x}^2 J(x,t) dx dt$.
\end{theorem}

The second result concerns the existence of solutions for problem $(P)$. For a given Lipschitz function of order $\gamma$, $f\in\mathcal{C}^{0,\gamma}(\mathbb{R}^n)$, we denote by $[f]_{\gamma}$ the corresponding seminorm of $f$. In the next statement $\mathcal{C}$ denotes the space of continuous functions.

\begin{theorem}\label{thm:theorem2existence}
Assume that $J(x,t)$ satisfies (J1), (J2), (J3) and (J4). Set $\alpha=\sup \{\beta:\iint_{s\leq \beta}J(y,s)dy ds<1\}$. Let $f\in L^{\infty}(\mathbb{R}^n)$ be given. Then there exists one and only one solution
$u(x,t)$ of $(P)$ in the space $(\mathcal{C}\cap L^{\infty})(\mathbb{R}^{n+1}_{+})$. If $f\in (L^1\cap L^{\infty})(\mathbb{R}^n)$, then $\int_{\mathbb{R}^n}u(x,t) dx=\int_{\mathbb{R}^n}f(x) dx$ for every $t\geq 0$. In particular, if $f$ is a density function, so is $u(\cdot,t)$ for every $t\geq 0$. Moreover, if $f$ belongs to $(\mathcal{C}^{0,\gamma}\cap L^{\infty})(\mathbb{R}^n)$ we have that
\begin{equation}\label{eq:temperaturelipschitz}
\abs{u(x,t)-f(x)}\leq C [f]_{\gamma}
\end{equation}
for $(x,t)\in \mathbb{R}^n\times [0,\alpha]$ and some $C$ which does not depend on $f$.
\end{theorem}

The next result which is interesting by itself contains a maximum principle which shall be used in the proof of Theorem~\ref{thm:theorem4convergence}. Precisely, the supremum of the probability density function in the future of $\alpha=\sup \{\beta:\iint_{s\leq \beta}J(y,s)dy ds<1\}$ coincides with its supremum in $\mathbb{R}^n\times [0,\alpha]$.

\begin{theorem}\label{thm:maximumprinciple}
Let $J$ be a kernel satisfying \textit{(J1)}, \textit{(J2)}, \textit{(J3)}, and \textit{(J4)}. Let $w(x,t)$ be a bounded function defined in $\mathbb{R}^{n+1}_{+}$ such that
\begin{equation}\label{eq:meanvaluewithJ}
w(x,t)=\iint J(x-y,t-s) w(y,s) dy ds
\end{equation}
for $(x,t)\in \mathbb{R}^n\times [\alpha,+\infty)$. Then,
\begin{equation*}
\sup_{(x,t)\in \mathbb{R}^{n+1}_{+}}\abs{w(x,t)} = \sup_{(x,t)\in \mathbb{R}^n\times [0,\alpha]}\abs{w(x,t)}.
\end{equation*}
\end{theorem}

Let us proceed to state the fourth result of the paper.
\begin{theorem}\label{thm:theorem4convergence}
For each $H\in\mathscr{H}$ there exists $C>0$ such that for every $r>0$ and every $f\in (\mathcal{C}^{0,\gamma}\cap L^{\infty})(\mathbb{R}^n)$
\begin{equation*}
\norm{u(H_r,f)-u}_{L^{\infty}(\mathbb{R}^{n+1}_{+})}\leq C [f]_{\gamma} r^{\gamma},
\end{equation*}
where $u$ is the temperature in $\mathbb{R}^{n+1}_{+}$ given by $u(x,t)=(\mathcal{W}(\cdot,t)\ast f)(x)$.
\end{theorem}

\medskip
Let us finally remark that in \cite{CaSi14} the authors prove the H\"{o}lder regularity for solutions of the master equation associated to CTRW's.

\medskip
In Section~\ref{sec:nonlocalweaklimit} we prove the weak convergence of parabolic rescalings to a heat equation. In Section~\ref{sec:meanvaluelaplacian}
we show existence of solution for the Cauchy nonlocal problem. Section~\ref{sec:maximumprinciple} is devoted to prove the maximum principle contained in 
Theorem~\ref{thm:maximumprinciple}. Finally, Section~\ref{sec:proofconvergence}, deals with convergence of solutions of rescalings of \textit{(P)} to temperatures.

\section{Some space time nonlocal parabolic operators and their weak limit. Proof of Theorem~\ref{thm:theorem1limits}}\label{sec:nonlocalweaklimit}

For $0<r<1$, since $\iint J(y,s)dyds=1$, applying Taylor's
formula we get
\begin{align*}
\iint & J_{r}(x-y,t-s)\varphi(y,s)dyds-\varphi(x,t)\\
&= \iint J_{r}(x-y,t-s)(\varphi(y,s)-\varphi(x,t))dyds\\
&=\iint J_{r}(x-y,t-s)\left[\sum_{i=1}^{n}\frac{\partial\varphi}{\partial
x_{i}}(x,t)(y_i-x_i)+\frac{\partial\varphi}{\partial t}(x,t)(s-t) \right.\\
& \phantom{\iint J_{r}}\left.+\frac{1}{2}(y-x,s-t)D^2\varphi(x,t)(y-x,s-t)^{t}+R(y-x,s-t)\right] dy ds,
\end{align*}
where $D^2$ denotes the Hessian matrix of the second derivatives of
$\varphi$ with respect to $x$ and $t$ and
$\abs{R(x,t)}=O(|x|^{2}+t^{2})^{\frac{3}{2}}$.

The last integral in the above identities can be written as the
sums of the following seven terms,

\begin{align*}
I &=\sum_{i=1}^{n}\frac{\partial\varphi}{\partial x_{i}}(x,t)\pii\iint  (y_i-x_{i}) J_{r}(x-y,t-s) dy ds\pdd,\\
II&=\frac{\partial\varphi}{\partial t}(x,t)\pii\iint (s-t) J_{r}(x-y,t-s) dy ds\pdd,\\
III&=\sum_{ij=1, i\neq j}^{n}\frac{\partial^{2}\varphi}{\partial
x_{i}\partial x_{j}}(x,t)\pii\frac{1}{2}\iint  (y_i-x_{i})(y_j-x_{j})J_{r}(x-y,t-s)dyds\pdd,\\
IV&=\sum_{i=1}^{n}\frac{\partial^{2}\varphi}{\partial x_{i}^{2}}(x,t)\pii\frac{1}{2}\iint  (y_i-x_{i})^{2}J_{r}(x-y,t-s)dyds\pdd,\\
V&=\sum_{i=1}^{n} \frac{\partial^{2}\varphi}{\partial
x_{i}\partial t}(x,t)\pii\frac{1}{2}\iint  (y_i-x_{i})(s-t) J_{r}(x-y,t-s)dyds\pdd,\\
VI&=\frac{\partial^{2}\varphi}{\partial t^{2}}(x,t)\pii\frac{1}{2}\iint  (s-t)^{2}J_{r}(x-y,t-s)dyds\pdd,\\
\intertext{and}
VII&=\iint  J_{r}(x-y,t-s) R(y-x,s-t)dy ds.
\end{align*}
Since for $t$ fixed $J$ is radial as a function of $x$, then $I$,
$III$ and $V$ vanish. For the other four integrals we perform the
parabolic change of variables $(z,\zeta)=(\tfrac{x-y}{r},\tfrac{t-s}{r^2})$ to obtain
\begin{align*}
II&=\frac{\partial\varphi}{\partial t}(x,t)r^{2}\pii-\iint  \zeta J(z,\zeta) dz d\zeta\pdd\\
IV&=\sum_{i=1}^{n}\frac{\partial^{2}\varphi}{\partial x_{i}^2}(x,t)r^{2}\pii\frac{1}{2}\iint  z_{i}^{2}J(z,\zeta)dzd\zeta\pdd \\
VI&=\frac{\partial^{2}\varphi}{\partial t^{2}}(x,t)r^{4}\pii\frac{1}{2}\iint  \zeta^{2}J(z,\zeta)dzd\zeta\pdd\\
VII&=\iint  J(z,\zeta) R(rz,r^{2}\zeta)dzd\zeta.
\end{align*}
Finally, since, as a function of $r$ close to zero, $VI$ and $VII$ are of
order at least $r^{3}$, we see that
\begin{equation*}
\lim\limits_{r\tends 0}\frac{1}{r^{2}}\cii
(J_{r}-\delta)\ast\varphi\cdd (x,t)=\lim\limits_{r\tends
0}\pii\frac{II}{r^{2}}+\frac{IV}{r^{2}} \pdd=\mu
\frac{\partial\varphi}{\partial t}(x,t)+\nu\triangle \varphi(x,t),
\end{equation*}
where $\mu$ and $\nu$ are defined as in the statement of Theorem~\ref{thm:theorem1limits}. That thus convergence is uniform in $\zR^{n+1}$ follows from the fact that $\varphi$ is a Schwartz function and so $VI$ and $VII$ converge to zero uniformly.

\begin{lemma}
For $J=\mathcal{H}$ we have that $\mu=-\nu$ and the limit equation in Theorem~\ref{thm:theorem1limits} is the heat equation multiplied by a constant.
\end{lemma}

\begin{proof}
All we need to show is that
\begin{equation}\label{igualdadconstantes}
\iint  \mathcal{H}(y,s)s\ dy ds=\frac{1}{2n}\iint
\mathcal{H}(y,s)|y|^{2}\ dy ds.
\end{equation}
Let us compute both of them in terms of the Euler gamma function and the area surface of the unit ball of $\mathbb{R}^{n}$, $S^{n-1}$.
On one hand we have that
\begin{align*}
\iint  \mathcal{H}(y,s)s\ dy ds
&=\frac{1}{4}\iint \mathcal{X}_{E(1)}(-y,-s)\frac{|y|^{2}}{s^{2}}s\ dyds\\
&=-\frac{1}{4}\iint _{E(1)}\frac{|y|^{2}}{s} dy
ds\\
&=\frac{1}{4}\int_{-\frac{1}{4\pi}}^{0}\int_{B\pii 0, \pii 2ns
\ln(4\pi (-s))\pdd^{\frac{1}{2}}\pdd}\frac{|y|^{2}}{-s} dy ds\\
&=\frac{1}{4}\int_{-\frac{1}{4\pi}}^{0}\frac{1}{-s}\int_{0}^{ \pii
2ns \ln(4\pi
(-s))\pdd^{\frac{1}{2}}}\rho^{n+1}\int_{S^{n-1}}d\sigma d\rho
ds\\
&=\frac{\sigma(S^{n-1})}{4(n+2)}\int_{-\frac{1}{4\pi}}^{0}\frac{1}{-s}\pii
2ns\ln(4\pi (-s))\pdd^{\frac{n+2}{2}}ds\\
&=\frac{\sigma(S^{n-1})}{4(n+2)}\int_{0}^{1}\frac{1}{t}\pii\frac{n}{2\pi}t(-\ln(t))\pdd^{\frac{n+2}{2}}dt\\
&=\frac{\sigma(S^{n-1})n^{\frac{n+2}{2}}}{4(n+2)2^{\frac{n+2}{2}}\pi^{\frac{n+2}{2}}}\int_{0}^{\infty}e^{-\theta\pii\frac{n+2}{2}\pdd}\theta^{\frac{n+2}{2}}d\theta\\
&=\frac{\sigma(S^{n-1})n^{\frac{n+2}{2}}}{4(n+2)2^{\frac{n+2}{2}}\pi^{\frac{n+2}{2}}}
\frac{2}{(n+2)}\frac{2^{\frac{n+2}{2}}}{(n+2)^{{\frac{n+2}{2}}}}\int_{0}^{\infty}e^{-\zeta}\zeta^{\frac{n+2}{2}}d\zeta\\
&=\frac{\sigma(S^{n-1})n^{\frac{n+2}{2}}}{2(n+2)^{\frac{n+6}{2}}\pi^{\frac{n+2}{2}}}\Gamma\pii\frac{n+4}{2}\pdd.
\end{align*}
On the other,
\begin{align*}
\frac{1}{2n}\iint  \mathcal{H}(y,s)|y|^{2}\ dy ds
&=\frac{1}{8n}\iint\mathcal{X}_{E(1)}(-y,-s)\frac{|y|^{2}}{s^{2}}|y|^{2}\
dyds\\
&=\frac{1}{8n}\iint _{E(1)}\frac{|y|^{4}}{s^{2}} dy ds\\
&=\frac{1}{8n}\int_{-\frac{1}{4\pi}}^{0}\int_{B\pii 0, \pii 2ns
\ln(4\pi (-s))\pdd^{\frac{1}{2}}\pdd}\frac{|y|^{4}}{s^{2}} dy ds\\
&=\frac{1}{8n}\int_{-\frac{1}{4\pi}}^{0}\frac{1}{s^{2}}\int_{0}^{
\pii 2ns \ln(4\pi
(-s))\pdd^{\frac{1}{2}}}\rho^{n+3}\int_{S^{n-1}}d\sigma d\rho
ds\\
&=\frac{\sigma(S^{n-1})}{8n(n+4)}\int_{-\frac{1}{4\pi}}^{0}\frac{1}{s^{2}}\pii
2ns\ln(4\pi (-s))\pdd^{\frac{n+4}{2}}ds\\
&=\frac{\sigma(S^{n-1})}{8(n+4)}\frac{n^{\frac{n+4}{2}}4\pi}{2^{\frac{n+4}{2}}\pi^{\frac{n+4}{2}}}\int_{0}^{1}\frac{1}{t^{2}}\pii t(-\ln(t))\pdd^{\frac{n+2}{2}}dt\\
&=\frac{\sigma(S^{n-1})n^{\frac{n+2}{2}}4\pi}{8(n+4)2^{\frac{n+4}{2}}\pi^{\frac{n+4}{2}}}\int_{0}^{\infty}e^{-\theta\pii\frac{n+2}{2}\pdd}\theta^{\frac{n+4}{2}}d\theta\\
&=\frac{\sigma(S^{n-1})n^{\frac{n+2}{2}}4\pi}{8(n+4)2^{\frac{n+4}{2}}\pi^{\frac{n+4}{2}}}\frac{2}{(n+2)}
\frac{2^{\frac{n+4}{2}}}{(n+2)^{\frac{n+4}{2}}}\int_{0}^{\infty}e^{-\zeta}\zeta^{\frac{n+4}{2}}d\zeta\\
&=\frac{\sigma(S^{n-1})n^{\frac{n+2}{2}}}{(n+4)(n+2)^{\frac{n+6}{2}}\pi^{\frac{n+2}{2}}}\Gamma\pii\frac{n+6}{2}\pdd.
\end{align*}
Now, since $\Gamma(z+1)=z\Gamma(z)$, we have that
$\frac{1}{n+4}\Gamma\pii\frac{n+6}{2}\pdd=\frac{1}{2}\Gamma\pii\frac{n+4}{2}\pdd$
and the proof is complete.
\end{proof}

\section{Existence of solutions for \textit{(P)}. Proof of Theorem~\ref{thm:theorem2existence}}\label{sec:meanvaluelaplacian}

Let $J(x,t)$ be a kernel defined in space time $\zR^{n+1}$ satisfying \textit{(J1)}, \textit{(J2)}, \textit{(J3)} and \textit{(J4)}. Let $f\in L^{\infty}(\mathbb{R}^{n})$ be given. Following the ideas in \cite{CERW08}, \cite{CoElRo09} and \cite{LibroRossi2010} we shall solve \textit{(P)} by iterated application of the Banach fixed point theorem. From \textit{(J3)} and \textit{(J4)}, $\alpha=\sup\{\beta:\iint_{s<\beta}J(x,s) dx ds<1\}$ is positive and finite. For the first step in the use of the fixed point theorem in the Banach space $\mathscr{B}_1=(\mathcal{C}\cap L^{\infty})(\mathbb{R}^n\times [0,\tfrac{\alpha}{2}])$ with the $L^{\infty}$ norm.

As in the statement of \textit{(P)} set
\begin{equation*}
\overline{v}(x,t)
= \left\{\begin{array}{ccc}
&f(x),\  & t<0;\\
&v(x,t),\  & t\in [0,\tfrac{\alpha}{2}],
\end{array}
\right.
\end{equation*}
where $v\in \mathscr{B}_1$. Since $\overline{v}$ is bounded on $\mathbb{R}^n\times (-\infty,\tfrac{\alpha}{2}]$ and $J\in L^1(\mathbb{R}^{n+1})$, the integral
\begin{equation*}
g(x,t):=\iint_{\mathbb{R}^n\times (-\infty,\tfrac{\alpha}{2}]}J(x-y,t-s)\overline{v}(y,s) dy ds
\end{equation*}
is absolutely convergent for $(x,t)\in \mathbb{R}^n\times [0,\tfrac{\alpha}{2}]$. Let us prove that, as a function of $(x,t)\in \mathbb{R}^n\times [0,\tfrac{\alpha}{2}]$ the function $g$ belongs to $\mathscr{B}_1$. In fact, from the definition of $g$ we see that
\begin{equation*}
\abs{g(x,t)}\leq \left(\iint J dy ds\right)\norm{\overline{v}}_{\infty}
\leq \sup \{\norm{f}_{\infty},\norm{v}_{\infty}\}.
\end{equation*}
Let us check the continuity of $g$. For $h\in \mathbb{R}^n$ and $k\in \mathbb{R}$ such that $(x+h,t+k)\in (-\infty,\tfrac{\alpha}{2}]$, we have that
\begin{align*}
\abs{g(x+h,t+k)-g(x,t)}&\leq \iint \abs{J(x+h-y,t+k-s)-J(x-y,t-s)}\abs{\overline{v}(y,s)} dy ds\\
&\leq \omega_1(\sqrt{\abs{h}^2+k^2})\norm{\overline{v}}_{\infty},
\end{align*}
where $\omega_1$ is the modulus of continuity in $L^1$ of $J$. Hence for $v\in \mathscr{B}_1$ we also have that $g\in \mathscr{B}_1$ when restricted to the strip $\mathbb{R}^n\times [0,\tfrac{\alpha}{2}]$.

Define $T_1: \mathscr{B}_1\to \mathscr{B}_1$ by $T_1v=g$.

Let us now prove that $T_1$ is a contractive mapping in $\mathscr{B}_1$. Let $v$ and $w$ be two functions in $\mathscr{B}_1$. Let $(x,t)\in \mathbb{R}^n\times [0,\tfrac{\alpha}{2}]$. Then, with
\begin{equation*}
\overline{w}(x,t)
= \left\{\begin{array}{ccc}
&f(x),\  & t<0;\\
&w(x,t),\  & t\in [0,\tfrac{\alpha}{2}],
\end{array}
\right.
\end{equation*}
we have that
\begin{align*}
T_1 v(x,t)-T_1 w(x,t)&= \iint_{s\leq\tfrac{\alpha}{2}} J(x-y,t-s)(\overline{v}(y,s)-\overline{w}(y,s)) dy ds\\
&=\iint_{0<s\leq\tfrac{\alpha}{2}} J(x-y,t-s)(v(y,s)-w(y,s)) dy ds.\\
\end{align*}
Hence
\begin{align*}
\norm{T_1 v-T_1 w}_{\infty}&\leq\left(\sup_{(x,t)\in \mathbb{R}^n\times [0,\tfrac{\alpha}{2}]} \iint_{0<s\leq\tfrac{\alpha}{2}} J(x-y,t-s) dy ds\right)\norm{v-w}_{\infty}
\end{align*}
Now, from the definition of $\alpha$, and \textit{(J1)}
\begin{align*}
\iint_{0<s\leq\tfrac{\alpha}{2}} J(x-y,t-s) dy ds &= \iint_{t-\tfrac{\alpha}{2}<\sigma\leq t} J(z,\sigma) dz d\sigma\\
&= \iint_{0<\sigma\leq t} J(z,\sigma) dz d\sigma\\
&\leq \iint_{0<\sigma\leq\tfrac{\alpha}{2}} J(z,\sigma) dz d\sigma=:\tau<1.
\end{align*}
So that $\norm{T_1 v - T_1 w}_{\infty}\leq \tau \norm{v-w}_{\infty}$.
Hence $T_1$ is a contractive mapping in $\mathscr{B}_1$. So that there exists a fixed point $u_1\in\mathscr{B}_1$ for $T_1$; $T_1u_1=u_1$. In other words
\begin{equation}
u_1(x,t)=\iint J(x-y,t-s) \overline{u_1}(y,s) dy ds
\end{equation}
for $x\in \mathbb{R}^n$ and $0\leq t\leq \tfrac{\alpha}{2}$.

Let us  check that
\begin{equation*}
\int_{\mathbb{R}^n} u_1(x,t) dx=\int_{\mathbb{R}^n} f(x) dx
\end{equation*}
for every $0\leq t\leq \tfrac{\alpha}{2}$, when $f\in L^1(\mathbb{R}^n)$. Since $u_1$ can be realized as the limit of the sequence of iterations of $T_1$ applied to any function $v\in \mathscr{B}_1$, we may take $v(x,t)=f(x)$ as the starting point. In doing so we see that the integral in variable $x$ of $\abs{T_1^{m}f(x,t)}$ is less and equal to $\int \abs{f} dx$. In fact, from \textit{(J3)}, we see
\begin{align*}
\int \abs{T_1f(x,t)} dx &= \int\abs{\iint J(x-y,t-s)f(y) dy ds} dx\\
&\leq \int \left(\iint J(x-y,t-s) dx ds\right) \abs{f(y)}dy\\
&= \int \abs{f} dy.
\end{align*}
Hence, inductively, assuming  $\int \abs{T_1^{m} f(x,t)}dx\leq \int \abs{f} dx$ we have
\begin{align*}
\int \abs{T_1^{m+1} f(x,t)}dx &= \int \abs{T_1(T_1^{m}f)(x,t)}dx\\
&= \int\abs{\iint J(x-y,t-s)\overline{T_1^{m}f}(y,s) dy ds} dx\\
&= \int\abs{\iint J(y,t-s)\overline{T_1^{m}f}(x-y,s) dy ds} dx\\
&= \int\left|\iint_{s<0} J(y,t-s)f(x-y) dy ds\right.  +\\
& \phantom{\iint_{s<0} ds} + \left.\iint_{s>0} J(y,t-s) T_1^{m}f(x-y,s) dy ds\right| dx\\
&\leq \iint_{s<0} J(y,t-s)\left|\int f(x-y)dx\right| dy ds +\\
& \phantom{\iint_{s<0}  ds} + \iint_{s>0} J(y,t-s) \left|\int T_1^{m}f(x-y,s) dx\right| dy ds\\
&\leq \int \abs{f} dx.
\end{align*}
With the same arguments we can conclude that $\int T^{m+1}_1 f(x,t) dx=\int f(x) dx$ for $0\leq t\leq\tfrac{\alpha}{2}$.
The result then follows since, for $f\in L^1\cap L^{\infty}$, we have that $T_1^m f$ tends to $u_1$ also in  $\mathcal{C}([0,\tfrac{\alpha}{2}],L^1(\mathbb{R}^n))$. In fact, if we prove that
\begin{equation}\label{eq:normoneCauchy}
|||T_1^{m+1}f- T_1^m f|||\leq \tau^m |||T_1^1 f-f|||
\end{equation}
where $|||v|||=\sup_{t\in [0,\tfrac{\alpha}{2}]}\norm{v(\cdot,t)}_{L^1(\mathbb{R}^n)}$, then $T_1^m f$ is also a Cauchy sequence in $\mathcal{C}([0,\tfrac{\alpha}{2}],L^1(\mathbb{R}^n))$. Since $T_1^m f$ converges uniformly to $u_1$ we get the desired preservation of the integral.
Let us prove \eqref{eq:normoneCauchy}.

Let us first check that $T_1^m f$ is continuous as a function of $t\in [0,\tfrac{\alpha}{2}]$ with values in $L^1(\mathbb{R}^n)$ for each $m$. Take $t$ and $t+h$ two points in $[0,\tfrac{\alpha}{2}]$. Then
\begin{align*}
&\int\abs{T_{1}^{m}f(x,t)-T_{1}^{m}f(x,t+h)}dx\\
&=\int\abs{\iint J(x-y,t-s)\overline{T_{1}^{m-1}f}(y,s)dyds\right.\\
&\phantom{=\int\iint}
\left.-\iint J(x-y,t+h-s)\overline{T_{1}^{m-1}f}(y,s)dyds}dx\\
&=\int\abs{\iint\langle  J(z,t-s)-J(z,t+h-s) \rangle  \overline{T_{1}^{m-1}f}(x-z,s) dz ds }dx\\
&\leq\int\int \abs{ J(z,t-s)-J(z,t+h-s) }\left(\int \abs{ \overline{T_{1}^{m-1}f}(x-z,s)}dx\right) dz ds\\
&\leq\int \abs{f(x)}dx\int\int \abs{ J(z,t-s)-J(z,t+h-s) }dz ds,
\end{align*}
which tends to zero when $h\to 0$ because $J\in L^1(\mathbb{R}^{n+1})$.

Similar calculations show that $T^m_1 f$ is a Cauchy sequence in the $|||\cdot|||$. In fact, for $t\in [0,\tfrac{\alpha}{2}]$ we have
\begin{align*}
&\int\abs{T_{1}^{m+1}f(x,t)-T_{1}^{m}f(x,t)}dx\\
&=\int\abs{\iint J(x-y,t-s) \left(  \overline{T_{1}^{m}f}(y,s)-\overline{T^{m-1}_1f}(y,s) \right)  dyds}dx\\
&=\int\abs{\int\int_{0\leq s \leq t} J(x-y,t-s) \left( T_{1}^{m}f(y,s)-T^{m-1}_1f(y,s)  \right) dyds }dx\\
&\leq\int\int_{0\leq s \leq t} J(z,t-s) \left(\int \abs{T_{1}^{m}f(x-z,s)-T^{m-1}_1f(x-z,s)} dx \right) dzds\\
&=\int\int_{0\leq s \leq t} J(z,t-s) \left(\int \abs{T_{1}^{m}f(x,s)-T^{m-1}_1f(x,s)} dx \right) dzds\\
&\leq\left(\sup\limits_{s\in [0,\tfrac{\alpha}{2}]}\int \abs{T_{1}^{m}f(x,s)-T^{m-1}_1f(x,s)} dx\right) \iint_{0\leq s \leq\tfrac{\alpha}{2}} J(z,t-s)dzds\\
&=\tau ||| T_{1}^{m}f-T^{m-1}_1f |||,
\end{align*}
hence
\begin{equation*}
|||T_{1}^{m+1}f-T^{m}_1f |||\leq \tau ||| T_{1}^{m}f-T^{m-1}_1 f|||.
\end{equation*}
%\begin{align*}
%&\norm{T_1^{m+1} f-T_1^m f}_{L^1\left(\mathbb{R}^n\times[0,\tfrac{\alpha}{2}]\right)} \\
%&= \iint_{0\leq t\leq\tfrac{\alpha}{2}}\abs{\iint_{0\leq s\leq t} J(x-y,t-s)\left(T_1^{m} f(y,s)-T_1^{m-1} f(y,s) dy ds\right)} dx dt\\
%&\leq \iint_{0\leq s\leq\tfrac{\alpha}{2}}\left(\iint_{s\leq t\leq \tfrac{\alpha}{2}} J(x-y,t-s)dx dt\right)\abs{T_1^{m} f(y,s)-T_1^{m-1} f(y,s)} dy ds\\
%&\leq \tau \norm{T_1^m f - T_1^{m-1} f}_{L^1\left(\mathbb{R}^n\times\left[0,\tfrac{\alpha}{2}\right]\right)}.
%\end{align*}
By iteration we obtain \eqref{eq:normoneCauchy}.

Let us observe that since $u_1(x,t)$ can be obtained as the iteration of $T_1$ starting at any function $v$ in $\mathscr{B}_1$, we can in particular take $v$ as the constant function
$\frac{s(f)-i(f)}{2}$, where $s(f)=\sup f$ and $i(f)=\inf f$. Then  $\overline{v}=v\mathcal{X}_{\{0\leq t\leq \tfrac{\alpha}{2}\}} + f\mathcal{X}_{\{t<0\}}$, so that $i(f)\leq \overline{v}\leq s(f)$ everywhere. From \textit{(J2)} and \textit{(J3)} we also have $i(f)\leq T_1 v\leq s(f)$ on $\mathbb{R}^n\times [0,\tfrac{\alpha}{2}]$. The same argument shows that for every iteration $T_1^{k}v$ of $T_1 v$ we have $i(f)\leq T_1^{k}v\leq s(f)$. Since $u_1$ is the uniform limit of $T_1^{k}v$ we get
\begin{equation*}
i(f)\leq u_1(x,t)\leq s(f)
\end{equation*}
on the strip $\mathbb{R}^n \times [0,\tfrac{\alpha}{2}]$.
So far we have existence and mass preservation for $t\in [0,\tfrac{\alpha}{2}]$.

Now proceed inductively by covering $\mathbb{R}^+$ with intervals of the type $[(i-1)\tfrac{\alpha}{2}, i\tfrac{\alpha}{2}]$. The first step, $i=1$ is precisely the one described above. Assume that $u_i\in \mathscr{B}_i=(\mathcal{C}\cap L^{\infty})(\mathbb{R}^n\times [(i-1)\tfrac{\alpha}{2}, i\tfrac{\alpha}{2}])$ for $i= 1,\ldots, j$ have been built in such a way that
\begin{equation*}
u_i(x,t)=\iint J(x-y,t-s) \overline{u_i}(y,s) dy ds,
\end{equation*}
with
\begin{equation*}
\overline{u_i}(x,t)
= \left\{\begin{array}{ccc}
&\overline{u_{i-1}}(x,t),\  & t<(i-1)\tfrac{\alpha}{2};\\
&u_i(x,t),\  & (i-1)\tfrac{\alpha}{2}\leq t\leq i\tfrac{\alpha}{2}.
\end{array}
\right.
\end{equation*}
Moreover, $\int_{\mathbb{R}^n} u_i(x,t) dx=\int_{\mathbb{R}^n}f(x) dx$ for $(i-1)\tfrac{\alpha}{2}\leq t\leq i\tfrac{\alpha}{2}$,
\begin{equation}\label{eq:infsupiteration}
i(f)\leq u_i(x,t)\leq s(f)
\end{equation}
for every $(x,t)\in \mathbb{R}^n\times [(i-1)\tfrac{\alpha}{2}, i\tfrac{\alpha}{2}]$, and $u_i(x,(i-1)\tfrac{\alpha}{2})=u_{i-1}(x,(i-1)\tfrac{\alpha}{2})$ for every $x$.

Define $\mathscr{B}_{j+1}$ as the space $(\mathcal{C}\cap L^{\infty})(\mathbb{R}^n\times [j\tfrac{\alpha}{2}, (j+1)\tfrac{\alpha}{2}])$ with the complete metric induced by the $L^{\infty}$ norm. For $v\in \mathscr{B}_{j+1}$, define
\begin{equation*}
T_{j+1}v(x,t)=\iint J(x-y,t-s)\overline{v}(y,s) dy ds
\end{equation*}
with
\begin{equation*}
\overline{v}(x,t)
= \left\{\begin{array}{ccc}
&\overline{u_{j}}(x,t),\  & t<j\tfrac{\alpha}{2};\\
&v(x,t),\  & j\tfrac{\alpha}{2}\leq t\leq (j+1)\tfrac{\alpha}{2}.
\end{array}
\right.
\end{equation*}
As in the case of $i=1$, it easy to check that with $(x,t)\in \mathbb{R}^n\times [j\tfrac{\alpha}{2}, (j+1)\tfrac{\alpha}{2}]$, $T_{j+1}v\in \mathscr{B}_{j+1}$. Hence $T_{j+1}: \mathscr{B}_{j+1} \to \mathscr{B}_{j+1}$. It is also easy to prove that $T_{j+1}$ is contractive on $\mathscr{B}_{j+1}$ with the same rate of contraction $\tau$ obtained when $i=1$.

Also, with the same argument as in the case $i=1$, with $\int \overline{u_{j}}(x,t)dx=\int f(x)dx$ when $t\leq j\tfrac{\alpha}{2}$ we have that for $t\in [j\tfrac{\alpha}{2}, (j+1)\tfrac{\alpha}{2}]$
\begin{equation*}
\int_{\mathbb{R}^n} u_{j+1}(x,t) dx
= \int_{\mathbb{R}} f(x) dx.
\end{equation*}

In order to check that $u_{j+1}(x,j\tfrac{\alpha}{2})=u_j(x,j\tfrac{\alpha}{2})$ we have to observe that for $j\tfrac{\alpha}{2}\leq t\leq (j+1)\tfrac{\alpha}{2}$, the fixed point equation is given by
\begin{equation*}
u_{j+1}(x,t)=\iint J(x-y,t-s) \overline{u_{j+1}}(y,s) dy ds.
\end{equation*}
For $t=j\tfrac{\alpha}{2}$, property \textit{(J1)} shows that the above integral only involves values of $s$ which are bounded above by $j\tfrac{\alpha}{2}$. Then, for those values of $s$, $\overline{u_{j+1}}(y,s)=\overline{u_{j}}(y,s)$. So that
\begin{equation*}
u_{j+1}(x,j\tfrac{\alpha}{2})=\iint J(x-y,j\tfrac{\alpha}{2}-s)\overline{u_j}(y,s) dy ds=u_j(x,j\tfrac{\alpha}{2}),
\end{equation*}
as desired.

Property \eqref{eq:infsupiteration} for $i=j+1$ can be proved following the same argument used in the case $i=1$. Let us notice that the function $u(x,t)$ defined in $\mathbb{R}^{n+1}_{+}$ by $u(x,t)=u_{j(t)}(x,t)$ with $j(t)$ the only positive integer for which $(j(t)-1)\tfrac{\alpha}{2}\leq t < j(t)\tfrac{\alpha}{2}$ is continuous and bounded. Moreover, $i(f)\leq u(x,t)\leq s(f)$ for every $(x,t)\in \mathbb{R}^{n+1}_{+}$.

The above remarks prove that $u\in \mathscr{B}=(\mathcal{C}\cap L^{\infty})(\mathbb{R}^{n+1}_{+})$ and solves \textit{(P)}.

In order to prove the uniqueness of the solution $u$ let us argue as follows. Assume that $u$ and $\widetilde{u}$ are two solutions. Then their restrictions on the strip $\mathbb{R}^{n}\times [0,\tfrac{\alpha}{2}]$ coincide. Since the fixed point of $T_{1}$ is unique and being a solution of $(P)$ in $\mathbb{R}^{n}\times [0,\tfrac{\alpha}{2}]$ is equivalent to be a fixed point for $T_{1}$, we see that $u\equiv\widetilde{u}$ on $\mathbb{R}^{n}\times [0,\tfrac{\alpha}{2}]$. For the next time interval $[\tfrac{\alpha}{2},\alpha]$ the restriction of both, $u$ and $\widetilde{u}$ to this interval are fixed points of the \textit{same} operator $T_{2}$. Again the uniqueness for the Banach fixed point guarantees $u\equiv \widetilde{u}$ on $\mathbb{R}^{n}\times [\tfrac{\alpha}{2},\alpha]$. Proceeding inductively we get that $u\equiv \widetilde{u}$ everywhere.

\medskip
Let us finally prove the estimate \eqref{eq:temperaturelipschitz}. First we shall show that \eqref{eq:temperaturelipschitz} holds true when $(x,t)\in\mathbb{ R}^n\times [0,\tfrac{\alpha}{2}]$. This can be accomplished because the function $u$ in the first time interval $[0,\tfrac{\alpha}{2}]$ coincides with $u_1$ which is provided by the Banach fixed point theorem and the rate of convergence can be estimated by the contraction constant $\tau$. We already know that $\tau=\iint_{s\leq \tfrac{\alpha}{2}}J(y,s) dy ds<1$. Set $u_1^m$ to denote the $m$-th iteration of $T_1$ applied to the initial guess $u_1^0=f$, then since $\norm{u_1^{m+1}-u_1^m}_{\infty}\leq \tau^m\norm{u_1^1-u_1^0}_{\infty}$, we see that
\begin{equation*}
\norm{u_1^m-f}_{\infty}\leq\left(\sum_{j=0}^m \tau^j\right)\norm{u_1^1-f}_{\infty}\leq\frac{1}{1-\tau}\norm{u_1^1-f}_{\infty}
\end{equation*}
for every $m=1,2,\ldots$

Let us now show that for $(x,t)\in \mathbb{R}^n\times [0,\tfrac{\alpha}{2}]$ there exists a constant $\widetilde{C}$ depending only on $J$ such that $\norm{u_1^1-f}_{\infty}\leq\widetilde{C}[f]_{\gamma}$. In fact,
\begin{align*}
\abs{u_1^1(x,t)-f(x)} &= \abs{(T_1f)(x,t)-f(x)}\\
&= \abs{\iint J(x-y,t-s)(f(y)-f(x))dy ds}\\
&\leq [f]_{\gamma}\left(\iint J(x-y,t-s)\abs{x-y}^{\gamma} dy ds\right).
\end{align*}
Hence for every $m=1,2,\ldots$ we have
\begin{equation*}
\norm{u_1^m - f}_{L^{\infty}\left(\mathbb{R}^n\times [0,\tfrac{\alpha}{2}]\right)}\leq C[f]_{\gamma},
\end{equation*}
where $C$ depends only on $J$. The same is true for the uniform limit $u_1$ of the sequence $u_1^m$, in other words
\begin{equation}\label{eq:firstestimationlipschitz}
\norm{u_1 - f}_{L^{\infty}\left(\mathbb{R}^n\times [0,\tfrac{\alpha}{2}]\right)}\leq C[f]_{\gamma}.
\end{equation}
Let us now check how to get the same type of estimate for the time interval $[\tfrac{\alpha}{2},\alpha]$. From the
construction of $u$ we have that on $\mathbb{R}^n\times[\tfrac{\alpha}{2},\alpha]$, $u=u_2$ with
\begin{equation*}
u_2(x,t)=\iint J(x-y,t-s) \overline{u_2}(y,s) dy ds
\end{equation*}
\begin{displaymath}
\overline{u_2}(x,t) = \left\{ \begin{array}{cc}
f(x), & t<0;\\
u_{1}(x,t), & t\in [0,\tfrac{\alpha}{2}];\\
u_{2}(x,t), & t\in [\tfrac{\alpha}{2},\alpha].
\end{array} \right.
\end{displaymath}
On $\mathbb{R}^n\times[\tfrac{\alpha}{2},\alpha]$ the solution $u_2$ is the only fixed point for the operator $T_2$ and, since the
limit $u_2$ of iterations $u_2^m$ of $T_2 u_2^0 = u_2^1$ is independent of the starting point $u_2^0$, let us take again $u_2^0 = f$.
Hence $\norm{u_2-f}_{\infty}\leq\frac{1}{1-\tau}\norm{u_2^1-f}_{\infty}$. Notice that if we write
$\overline{f}(y,s)= f(y)\mathcal{X}_{s<0}(s)+u_1(y,s)\mathcal{X}_{[0,\tfrac{\alpha}{2}]}(s)+f(y)\mathcal{X}_{[\tfrac{\alpha}{2},\alpha]}(s)$
\begin{equation*}
u_2^1 (x,t)=\iint J(x-y,t-s) \overline{f}(y,s) dy ds.
\end{equation*}
Let us finally check that the desired estimate holds for $\norm{u_2^1 - f}_{\infty}$ in $\mathbb{R}^n\times[\tfrac{\alpha}{2},\alpha]$.
Take $(x,t)\in\mathbb{R}^n\times[\tfrac{\alpha}{2},\alpha]$, then
\begin{align*}
\abs{u_2^1(x,t)-f(x)} &= \abs{(T_2f)(x,t)-f(x)}\\
&= \abs{\iint J(x-y,t-s) \overline{f}(y,s) dy ds - f(x)}\\
&\leq \iint_{s\leq 0} J(x-y,t-s) \abs{f(y) - f(x)}dy ds\\
&\phantom{\leq \iint_{s\leq 0}} + \iint_{0<s<\tfrac{\alpha}{2}} J(x-y,t-s) \abs{u_1(y,s)-f(y)}dy ds\\
&\phantom{\leq \iint_{s\leq 0J(x-y,t-s)}} +  \iint_{s\leq\alpha} J(x-y,t-s) \abs{f(y) - f(x)}dy ds.
\end{align*}
The first and the third terms in the right hand side of the above inequality are bounded by the product
of the Lip$\gamma$ seminorm of $f$ and a constant depending only on $J$. For the second term we use \eqref{eq:firstestimationlipschitz} and
we are done.

\section{Maximum Principle: Proof of Theorem~\ref{thm:maximumprinciple}}\label{sec:maximumprinciple}

Recall that $\alpha=\sup \{\beta:\iint_{s\leq \beta}J(y,s)dy ds<1\}$. Since the function $I(\beta)=\iint_{s\leq \beta} J dy ds$ is increasing and continuous as a function of $\beta$, $\alpha$ is also the infimum of those values of $\beta$ for which $I(\beta)=1$. Moreover, from definition of $\alpha$, $0<I(\tfrac{\alpha}{2})<1$.

Let $t_k=\alpha+(k-1)\tfrac{\alpha}{2}$, $B_k=\mathbb{R}^n\times [0,t_k]$, $S_k=\sup_{B_k}\abs{w}$ for $k=1,2,\ldots$. Let us see that $S_k=S_{k-1}$. Let $(x,t)\in \mathbb{R}^n\times [t_{k-1},t_k]$, hence

\begin{center}
  \begin{tikzpicture}[scale=1]
    %  Everything for 0<s<alpha
    \begin{scope}
        \shade[bottom color=black!40, top color=black!50]
            (0,1) rectangle (5,0);
    \end{scope}
    \draw[important line]  (0,1) node[left] {$\alpha$} -- (5,1);
    \begin{scope}
        \shade[bottom color=black!10, top color=black!10]
            (0,3) rectangle (5,2.5);
    \end{scope}
      \begin{scope}
        \shade[bottom color=black!20, top color=black!20]
            (0,2.5) rectangle (5,2);
    \end{scope}
    % Axis
    \coordinate (y) at (0,4);
    \coordinate (x) at (5,0);
    \draw[-] (y) node[above] {$s$} -- (0,0) --  (x) node[below]
    {$\mathit{\mathbb{R}^n}$};
    \draw[important line]  (0,3-0.5-0.5) node[left] {$t_{k-2}$} -- (5,3-0.5-0.5);
    \draw[important line]  (0,3-0.5) node[left] {$t_{k-1}$} -- (5,3-0.5);
    \draw[important line]  (0,3) node[left] {$t_k$} -- (5,3);
    \fill[black] (2.5,2.9) circle (1pt) node[above] {$(x,t)$};
    \fill [opacity=0.5,blue]
    (2.5,2.9) -- (2.65,2.83) -- (2.7,2.7) -- (2.9,2.45) -- (2.8,2.3) -- (2.7,2.15) -- (2.65,2.1) -- (2.6,2) -- (2.5,1.9) -- (2.45,2) -- (2.45,2.15) -- (2.4,2.2) -- (2.3,2.21) -- (2.2,2.37) -- (2.18,2.45) -- (2.24,2.56) -- (2.3,2.7) -- (2.4,2.85) -- (2.45,2.87) -- cycle;
%% Create random(ish) points
%mark=between positions 2.5 and 2.8 step 1pt;
%\foreach \i in {1,...,30}
%  \fill [opacity=0.5] (360/30*\i:1+rnd*1) circle [radius=.025] coordinate (mark-\i);
%% Join them up
%\fill [opacity=0.5,blue]
%  (mark-1) \foreach \i in {2,...,30}{ -- (mark-\i) } -- cycle;
  \end{tikzpicture}
  %\caption{..} PONER comentario de la FIGURA 2
\end{center}

\begin{align*}
&\abs{w(x,t)} = \abs{\iint J(x-y,t-s) w(y,s) dy ds}\\
&= \abs{\,\,\iint\limits_{t_{k-1}\leq s\leq t} J(x-y,t-s) w(y,s) dy ds + \iint\limits_{t-\alpha\leq s\leq t_{k-1}} J(x-y,t-s) w(y,s) dy ds}\\
&\leq S_k\iint\limits_{t_{k-1}\leq s\leq t} J(x-y,t-s) dy ds + S_{k-1}\iint\limits_{t-\alpha\leq s\leq t_{k-1}} J(x-y,t-s) dy ds\\
&= S_k\left(1-\!\!\!\!\!\!\iint\limits_{t-\alpha\leq s\leq t_{k-1}}\!\!\!\!\!\!J(x-y,t-s) dy ds\right) + S_{k-1}\left(\,\,\iint\limits_{t-\alpha\leq s\leq t_{k-1}}\!\!\!\!\!\! J(x-y,t-s) dy ds\right)\\
&= S_k - (S_k - S_{k-1})\left(\iint_{t-\alpha\leq s\leq t_{k-1}} J(x-y,t-s) dy ds\right).
\end{align*}
Hence
%\begin{equation*}
%(S_k-S_{k-1})\left(\,\,\iint_{B_{k-1}\cap \{s\geq t-\alpha\}} J(x-y,t-s) dy ds\right)\leq S_k-\abs{w(x,t)},
%\end{equation*}
%and
\begin{equation*}
\left(\iint_{t-\alpha\leq s\leq t_{k-1}} J(x-y,t-s) dy ds\right)\left( S_k-S_{k-1}\right) \leq S_k-\abs{w(x,t)},
\end{equation*}
since
\begin{align*}
\iint_{t-\alpha\leq s\leq t_{k-1}} J(x-y,t-s) dy ds &= \iint_{t-t_{k-1}\leq s_1\leq\alpha} J(x-y,s_1) dy ds_1\\
& \geq \iint_{\tfrac{\alpha}{2}\leq s_1\leq\alpha} J(x-y,s_1) dy ds_1
=1- I(\tfrac{\alpha}{2}).
\end{align*}
Then
\begin{equation}\label{eq:supremumbandas}
(1- I(\tfrac{\alpha}{2}))(S_k-S_{k-1})\leq S_k-\sup_{B_k\setminus B_{k-1}}\abs{w}
\end{equation}
Of course $S_k\geq S_{k-1}$, if $S_k>S_{k-1}$, then $S_k= \sup_{B_k\setminus B_{k-1}}\abs{w}$ and the right hand of \eqref{eq:supremumbandas} equals  zero. Since $1- I(\tfrac{\alpha}{2})$ is positive we have that $S_k\leq S_{k-1}$, a contradiction.

\section{Convergence of solutions. Proof of  Theorem~\ref{thm:theorem4convergence}}\label{sec:proofconvergence}

Along this section we shall assume that $J$ in problem \textit{(P)} is a fixed parabolic rescaling of a mean value kernel $H\in\mathscr{H}$. More precisely for $r>0$ we consider the CTRW, with past density distribution $f(x)$, generated by the space time probability density $H_r(x,t)=\tfrac{1}{r^{n+2}}H(\tfrac{x}{r},\tfrac{t}{r^2})$, with $H(x,t)$ satisfying \textit{(J1)} to \textit{(J5)} and the mean value formula. The space probability density function for the localization of the moving particle at time $t$ is the solution $u(H_r,f)(\cdot,t)$ of $P(H_r,f)$.

The result contained in Theorem~\ref{thm:theorem4convergence} will be a consequence of the two following lemmas and the maximum principle contained in Theorem~\ref{thm:maximumprinciple}.
\begin{lemma}\label{lem:infinitynormsolutionPrandf}
Let $J$ be a kernel satisfying \textit{(J1)}, \textit{(J2)}, \textit{(J3)} and \textit{(J4)}. Set $\alpha =\sup \{\beta:\iint_{s\leq \beta}J(y,s) dyds <1\}$. Let $f\in (\mathcal{C}^{0,\gamma}\cap L^{\infty})(\mathbb{R}^{n})$, $0<\gamma\leq 1$. Then there exists a constant $C>0$ such that for every $r>0$
\begin{equation*}
\norm{u(J_r,f)-f}_{L^{\infty}(\mathbb{R}^n\times [0,\alpha r^2])}\leq C [f]_{\gamma} r^{\gamma}.
\end{equation*}
\end{lemma}

The next lemma contains a well known result on the rate of convergence of temperatures to the initial condition when it belongs to a Lipschitz class.

\begin{lemma}\label{lem:convergetemperatureLipschitz}
Let $f\in (\mathcal{C}^{0,\gamma}\cap L^{\infty})(\mathbb{R}^{n})$ and $u(x,t)=(4\pi t)^{-\tfrac{n}{2}}\int_{\mathbb{R}^n}e^{-\tfrac{\abs{x-y}^2}{4t}}f(y) dy$ the solution of
\begin{equation*}
\left\{\begin{array}{cc}
\frac{\partial u}{\partial t}=\triangle u,\  & \mathbb{R}^{n+1}_{+};\\
u(x,0)=f(x),\  & x\in \mathbb{R}^n.
\end{array}
\right.
\end{equation*}
Then there exists a constant $C>0$ such that
\begin{equation*}
\abs{u(x,t)-f(x)}\leq C [f]_{\gamma} t^{\tfrac{\gamma}{2}}
\end{equation*}
for every $(x,t)\in \mathbb{R}^{n+1}_{+}$.
\end{lemma}

\begin{proof}[ Proof of Theorem~\ref{thm:theorem4convergence}]
From Lemma~\ref{lem:infinitynormsolutionPrandf} and Lemma~\ref{lem:convergetemperatureLipschitz} we have that
\begin{equation*}
\sup_{(x,t)\in \mathbb{R}^n\times [0,\alpha r^2]}\abs{u(H_r,f)(x,t)-u(x,t)}\leq C [f]_{\gamma} r^{\gamma}.
\end{equation*}
Now, since $\alpha=\sup \{\beta:\iint_{s\leq \beta}H(y,s)dy ds<1\}$ we also have that $\alpha r^2=\sup \{\beta:\iint_{s\leq \beta}H_r(y,s)dy ds<1\}$. Hence for $(x,t)\in \mathbb{R}^n\times (\alpha r^2,+\infty)$ we have that the support of $H_r(x-y,t-s)$ as a function of $(y,s)$ is contained in $\mathbb{R}^{n+1}_{+}$. So that for a temperature $u$ defined on $\mathbb{R}^{n+1}_{+}$ and $t>\alpha r^2$, since $H\in\mathscr{H}$, the mean value formula holds and
\begin{equation*}
u(x,t)=\iint H_r(x-y,t-s) u(y,s) dy ds
\end{equation*}
for $x\in \mathbb{R}^n$ and $t>\alpha r^2$.

On the other hand, we also have that $u(H_r,f)=H_r\ast u(H_r,f)$ for $x\in \mathbb{R}^n$ and $t>\alpha r^2$, because $u(H_r,f)$ solves $P(H_r,f)$. Hence, applying Theorem~\ref{thm:maximumprinciple} with $H_r$ instead of $J$, $\alpha r^2$ instead of $\alpha$, $u(H_r,f) - u$ instead of $w$, we get
\begin{equation*}
\sup_{\mathbb{R}^{n+1}_{+}}\abs{u(H_r,f)(x,t)-u(x,t)}\leq  \sup_{\mathbb{R}^n\times [0,\alpha r^2]}\abs{u(H_r,f)(x,t)-u(x,t)}\leq C [f]_{\gamma} r^{\gamma},
\end{equation*}
as desired.
\end{proof}

\begin{proof}[Proof of Lemma~\ref{lem:infinitynormsolutionPrandf}]
The result follows from \eqref{eq:temperaturelipschitz} in Theorem~\ref{thm:theorem2existence} after parabolic rescaling. In fact, set $u(J,g)$ to denote the solution in $\mathbb{R}^{n+1}_{+}$ of $P(J,g)$. With this notation we have that
\begin{equation*}
u(J_r,f)=\left[u\left(J,f_{\tfrac{1}{r}}\right)\right]_r,
\end{equation*}
for each $r>0$. Hence, for $x\in \mathbb{R}^n$ and $0\leq\tfrac{t}{r^2}\leq\alpha$, from \eqref{eq:temperaturelipschitz} with $f(r\cdot)$ instead $f$, we have
\begin{align*}
\abs{u(J_r,f)(x,t)-f(x)}&=\abs{\left[u\left(J,f_{\tfrac{1}{r}}\right)\right]_r(x,t)-f(x)}\\
&=\abs{u(J,f(r\cdot))\left(\tfrac{x}{r},\tfrac{t}{r^2}\right)-f\left(r\left(\tfrac{x}{r}\right)\right)}\\
&\leq C [f(r\cdot)]_{\gamma}=Cr^{\gamma}[f]_{\gamma}.
\end{align*}
\end{proof}

\begin{proof}[Proof of Lemma~\ref{lem:convergetemperatureLipschitz}]
For $(x,t)\in \mathbb{R}^{n+1}_{+}$, since the Weierstrass kernel has integral equal to one, we have
\begin{align*}
\abs{u(x,t)-f(x)}&=\abs{\frac{1}{(4\pi t)^{n/2}}\int_{\mathbb{R}^{n}}e^{-\frac{|x-y|^{2}}{4t}}f(y) dy-f(x)}\\
&\leq\frac{1}{(4\pi t)^{n/2}}\int_{\mathbb{R}^{n}}e^{-\frac{|x-y|^{2}}{4t}}\abs{f(y)-f(x)} dy\\
&\leq\frac{[f]_{\gamma}}{(4\pi t)^{n/2}}\int_{\mathbb{R}^{n}}e^{-\frac{|x-y|^{2}}{4t}}\abs{y-x}^{\gamma} dy\\
&=\frac{[f]_{\gamma}}{\pi^{n/2}}\int_{\mathbb{R}^{n}}e^{-\abs{z}}\abs{z}^{\gamma}dz\, t^{\frac{\gamma}{2}}.
\end{align*}

\end{proof}

%\bibliographystyle{amsplain}
%\bibliography{ref}

\begin{thebibliography}{1}

\bibitem{LibroRossi2010}
F. Andreu-Vaillo, J.~M. Maz{\'o}n, J.~D. Rossi, and
  J.~J. Toledo-Melero, \emph{Nonlocal diffusion problems}, Mathematical
  Surveys and Monographs, vol. 165, American Mathematical Society, Providence,
  RI; Real Sociedad Matem\'atica Espa\~nola, Madrid, 2010. 
  %\MR{2722295(2011i:35002)}

\bibitem{CaSi14}
L. Caffarelli and L. Silvestre, \emph{H\"{o}lder regularity for generalized
  master equations with rough kernels}, Advances in Analysis: The Legacy of
  Elias M. Stein, Princeton University Press, 2014.

\bibitem{CoElRo09}
C. Cortazar, M. Elgueta, and J.~D. Rossi, \emph{Nonlocal diffusion
  problems that approximate the heat equation with {D}irichlet boundary
  conditions}, Israel J. Math. \textbf{170} (2009), 53--60. 
  %\MR{2506317 (2010e:35197)}

\bibitem{CERW08}
C. Cortazar, M. Elgueta, J.~D. Rossi, and N. Wolanski, \emph{How
  to approximate the heat equation with {N}eumann boundary conditions by
  nonlocal diffusion problems}, Arch. Ration. Mech. Anal. \textbf{187} (2008),
  no.~1, 137--156. %\MR{2358337 (2008k:35261)}

\bibitem{Evans98}
L.~C. Evans, \emph{Partial differential equations}, Graduate Studies in
  Mathematics, vol.~19, American Mathematical Society, Providence, RI, 1998.
  %\MR{1625845 (99e:35001)}

\bibitem{MeKla2000}
R. Metzler and J. Klafter, \emph{The random walk's guide to anomalous
  diffusion: a fractional dynamics approach}, Phys. Rep. \textbf{339} (2000),
  no.~1, 77. %\MR{1809268 (2001k:82082)}

\bibitem{Watson2012}
N.~A. Watson, \emph{Introduction to heat potential theory}, Mathematical
  Surveys and Monographs, vol. 182, American Mathematical Society, Providence,
  RI, 2012. %\MR{2907452}

\end{thebibliography}

\def\cprime{$'$}
\providecommand{\bysame}{\leavevmode\hbox to3em{\hrulefill}\thinspace}
\providecommand{\MR}{\relax\ifhmode\unskip\space\fi MR }
\providecommand{\MRhref}[2]{%
  \href{http://www.ams.org/mathscinet-getitem?mr=#1}{#2}
}
\providecommand{\href}[2]{#2}

%% Address Authors

\bigskip
\noindent{\footnotesize
\textsc{Instituto de Matem\'atica Aplicada del Litoral, CONICET, UNL.}

%%\textsc{Departamento de Matem\'atica, Facultad de Ingenier\'ia Qu\'imica, UNL}
%
%
\smallskip
\noindent\textmd{CCT CONICET Santa Fe, Predio ``Alberto Cassano'', Colectora Ruta Nac.~168 km 0, Paraje El Pozo, 3000 Santa Fe, Argentina.}
}
%
%\bigskip

\end{document}